\documentclass[12pt]{amsart}

\usepackage{amssymb}
\usepackage{enumerate}
\usepackage{amscd}

\newtheorem{theorem}{Theorem}[section]
\newtheorem{lemma}[theorem]{Lemma}
\newtheorem{corollary}[theorem]{Corollary}
\newtheorem{proposition}[theorem]{Proposition}

\newcommand{\ZZ}{\mathbb{Z}}

\newcommand{\groupsum}[2]{\mathcal{O}_{#1}(#2)}
\newcommand{\groupsumvoid}[1]{\mathcal{O}_{#1}}
\newcommand{\oriented}[1]{{#1}^{or}}

\newcommand{\BD}[2]{\mathsf{BD}_{#1}^{#2}}
\newcommand{\M}[1][{}]{\mathsf{M}_{#1}}
\newcommand{\Symm}[1]{\mathfrak{S}_{#1}}
\newcommand{\zivaljevic}{$\check{\mathrm{Z}}$ivaljevi\'c}

\begin{document}

  \author{Jakob Jonsson}
  \title[Five-Torsion in the
  Matching Complex on 14 Vertices]{Five-Torsion in the
    Homology of the Matching Complex on $14$ Vertices}
  \date{\today}
  \thanks{Research supported by European Graduate
      Program ``Combinatorics, Geometry, and Computation'', DFG-GRK
      588/2. }  

  \begin{abstract}
    J. L. Andersen proved that
    there is $5$-torsion in the bottom nonvanishing homology group of
    the simplicial complex of graphs of degree at
    most two on seven vertices. We use this result to demonstrate that
    there is $5$-torsion 
    also in the bottom nonvanishing homology group of the matching
    complex $\M[14]$ on $14$ vertices. 
    Combining our observation with results due to Bouc
    and to Shareshian and Wachs, we 
    conclude that the case $n=14$ is exceptional; for all other $n$,
    the torsion subgroup of the bottom nonvanishing homology group 
    has exponent three or is zero. 
    The possibility remains that there is other torsion than
    $3$-torsion in higher-degree homology groups of $\M[n]$ when $n
    \ge 13$ and $n \neq 14$.
  \end{abstract}

  \maketitle

  \noindent
  This is a preprint version of a paper published in
  {\em Journal of Algebraic Combinatorics} {\bf 29} (2009), no. 1,
  81--90.

  \section{Introduction}
  \label{intro-sec}

  Throughout this note, by a graph we mean a finite graph with loops
  allowed but with no multiple edges or multiple loops. The degree of
  a vertex $i$ in a given graph $G$ is the number of times $i$ appears
  as an endpoint of an edge in $G$; thus a loop at $i$ (if present)
  is counted twice, whereas other edges containing $i$ are counted
  once.

  Given a family $\Delta$ of graphs on a fixed vertex set, we 
  identify each member of $\Delta$ with its edge set. In particular,
  if $\Delta$ is closed under deletion of edges, then 
  $\Delta$ is an abstract simplicial complex.
  Let $n \ge 1$ and let $\lambda = (\lambda_1, \ldots, \lambda_n)$ be
  a sequence of nonnegative integers. We define
  $\BD{n}{\lambda}$ to be the simplicial complex of graphs on the
  vertex set $[n] := \{1, \ldots, n\}$ such that the degree of 
  the vertex $i$ is at most $\lambda_i$. 
  We write $\BD{n}{k} := \BD{n}{(k, \ldots, k)}$
  and $\M[n] := \BD{n}{1}$; the latter complex is the \emph{matching
    complex} on $n$ vertices.

  The
  topology of $\M[n]$ and related complexes has been subject to
  analysis
  in several theses \cite{Andersen,Dong,Garst,thesis,thesislmn,Kara,Kson}
  and papers 
  \cite{Ath,BBLSW,BLVZ,Bouc,DongWachs,FH,KRW,RR,ShWa,Ziegvert}; see
  Wachs \cite{Wachs} for an excellent survey and further references.

  The prime $3$ is known to play a prominent part in the homology of
  $\M[n]$. Specifically, write $\nu_n = \lfloor
  \frac{n-2}{3}\rfloor = \lceil\frac{n-4}{3}\rceil$. By a result due
  to Bj\"orner, Lov\'asz, Vre\'cica, 
  and {\zivaljevic} 
  \cite{BLVZ}, the reduced homology group $\tilde{H}_i(\M[n]; \ZZ)$ is
  zero whenever $i < \nu_n$. Bouc \cite{Bouc} showed that 
  $\tilde{H}_{\nu_n}(\M[n]; \ZZ) \cong \ZZ_3$ whenever $n = 3k+1 \ge
  7$ and that 
  $\tilde{H}_{\nu_n}(\M[n]; \ZZ)$ has exponent dividing nine
  whenever $n = 3k \ge 12$. Shareshian and Wachs
  extended and improved Bouc's result:
  \begin{theorem}[Shareshian and Wachs \cite{ShWa}]
    \label{ShWa-thm}
    $\tilde{H}_{\nu_n}(\M[n]; \ZZ)$ is an elementary $3$-group 
    for $n \in \{7,10,12,13\}$ and also for $n \ge 15$.
    The torsion subgroup of $\tilde{H}_{\nu_n}(\M[n]; \ZZ)$
    is again an elementary $3$-group for $n \in \{9,11\}$ and  
    zero for $n \in \{1,2,3,4,5,6,8\}$. For the remaining case $n=14$, 
    $\tilde{H}_{\nu_n}(\M[n]; \ZZ)$ is a finite group with
    nonvanishing $3$-torsion. 
  \end{theorem}
  To prove that the group $\tilde{H}_{\nu_n}(\M[n]; \ZZ)$ is
  elementary for $n \equiv 0 \pmod{3}$ and $n \ge 12$ and 
  for $n \equiv 2 \pmod{3}$ and $n \ge 17$, 
  Shareshian and Wachs relied on a computer
  calculation of the group $\tilde{H}_{3}(\M[12];\ZZ)$. 
  The existence of $3$-torsion in the homology of $\M[9]$ and $\M[11]$
  also relied on such calculations. Unfortunately, attempts to stretch
  this computer approach beyond $n=12$ have failed; the size of
  $\M[n]$ is too large for the existing software to handle
  when $n \ge 13$. In particular, the structure of the bottom
  nonvanishing homology group of $\M[14]$ has remained a mystery. 

  For completeness, let us mention that there is $3$-torsion also in
  higher-degree homology groups \cite{bettimatch}. 
  More precisely, there is
  $3$-torsion in $\tilde{H}_d(\M[n]; \ZZ)$ whenever $\nu_n \le d \le
  \frac{n-6}{2}$. 
  As a consequence, since there is homology in degree
  $\lfloor\frac{n-3}{2}\rfloor$ but not above this degree \cite{Bouc},
  there is $3$-torsion in almost all nonvanishing homology groups of
  $\M[n]$, the only exceptions being the top degree
  $\lfloor\frac{n-3}{2}\rfloor$ 
  and \emph{possibly} the degree $\frac{n-5}{2}$ just
  below it for odd $n$. The homology in the latter degree is known to
  contain $3$-torsion for $n \in \{7,9,11,13\}$; see 
  Proposition~\ref{m13-prop} for the case $n = 13$.
  A complete description of the homology groups 
  of $\M[n]$ is known only for $n \le 12$; see
  Table~\ref{matching-fig}.

  The appearance of $3$-torsion being so
  prominent, it makes sense to ask whether this is the {\em only}
  kind of torsion that appears in $\M[n]$ for {\em any} given
  $n$. Indeed, Babson, Bj\"orner, Linusson, Shareshian, and Welker 
  \cite{BBLSW} asked this very question. 
  Based on the overwhelming evidence presented in
  Theorem~\ref{ShWa-thm}, Shareshian and
  Wachs \cite{ShWa} conjectured that
  $\tilde{H}_{\nu_{14}}(\M[14]; \ZZ) = \tilde{H}_{4}(\M[14]; \ZZ)$ is
  an elementary $3$-group.
  Surprisingly, the conjecture turns out to be false:
  \begin{theorem}
    \label{main-thm}
    $\tilde{H}_{4}(\M[14]; \ZZ)$ is a finite group of exponent a
    multiple of $15$.
  \end{theorem}
  Theorem~\ref{main-thm} being just one specific example, the question
  of Babson et al$.$ remains unanswered in general; we do
  not know whether there is other torsion than $3$-torsion 
  in the homology
  of other matching complexes. See Section~\ref{further-sec} for some
  discussion.

  To prove Theorem~\ref{main-thm}, we use a result due to Andersen
  about the homology of $\BD{7}{2}$:
  \begin{theorem}[Andersen \cite{Andersen}]
    \label{andersen-thm}
    We have that 
        \[
        \tilde{H}_i(\BD{7}{2}; \ZZ) \cong 
        \left\{
        \begin{array}{ll}
                \ZZ_5 & \mbox{if } i=4; \\
                \ZZ^{732} & \mbox{if } i = 5;\\
                0 & \mbox{otherwise}.
        \end{array}
        \right. 
        \]
  \end{theorem}
  \noindent
  \emph{Remark.}
  We have verified Theorem~\ref{andersen-thm} using the {\tt Homology}
  computer program \cite{Homoprog}.

  Let $\Symm{n}$ be the symmetric group on $n$ elements.
  We relate Andersen's result to the homology of $\M[14]$ via 
  a map $\pi^*$ from $\tilde{H}_4(\M[14];\ZZ)$ to
  $\tilde{H}_4(\BD{7}{2};\ZZ)$; this map is 
  induced by the natural action on $\M[14]$ by the
  Young group $(\Symm{2})^7$. 
  Using a standard representation-theoretic argument, we construct
  an ``inverse'' $\varphi^*$ of $\pi^*$ with the property that 
  $\pi^*\circ \varphi^*(z) = |(\Symm{2})^7|\cdot z$ for all 
  $z \in \tilde{H}_4(\BD{7}{2};\ZZ)$. To conclude the proof, one
  observes that $\varphi^*(z)$ is nonzero unless the order of $z$
  divides the order of $(\Symm{2})^7$. Since 
  the latter order is $128$, the image under $\varphi^*$ of
  any nonzero element of order five is again a nonzero element of
  order five.

  Using computer, we have also been able to deduce that 
  there is $3$-torsion in the homology of 
  $\tilde{H}_4(\BD{7}{(2^31^7)}; \ZZ)$.
  where $(2^31^7)$ denotes the sequence $(2,2,2,1,1,1,1,1,1,1)$.
  An argument similar to the one above yields that
  $\tilde{H}_4(\M[13]; \ZZ)$ contains $3$-torsion. By the results of
  Bouc \cite{Bouc}, we already know that $\tilde{H}_3(\M[13]; \ZZ)
  \cong \ZZ_3$.

  \begin{table}
    \caption{The homology of $\M[n]$ for $n \le 14$; 
      see Wachs \cite{Wachs} for explanation of the parts that are
      not explained in the present note. $T_1$ and $T_2$ are
      nontrivial finite groups of exponent a multiple of $3$ and $15$,
      respectively; see Proposition~\ref{m13-prop} and
      Theorem~\ref{main-thm}.} 
    \begin{footnotesize}
    \begin{center}
        \begin{tabular}{|r||c|c|c|c|c|c|}
          \hline
          & & & & & &
          \\[-1.5ex]
          $\tilde{H}_i(\M[n];\ZZ)$
          &  $i=0$ & \ 1 \ & \ 2 \ & \ 3 \ & \ 4 \ & \ 5 \
          \\
          \hline
          \hline
          & & & & & &
          \\[-2ex]
          $n = 3$ & $\ZZ^2$ & - & - & - & - & - \\
          \hline
          & & & & & &
          \\[-2ex]
          $4$ & $\ZZ^2$ & - & - & - & - & - \\
          \hline
          & & & & & &
          \\[-2ex]
          $5$ & - & $\ZZ^6$ & - & - & - & - \\
          \hline
          & & & & & &
          \\[-2ex]
          $6$ & - & $\ZZ^{16}$ & - & - & - & - \\
          \hline
          & & & & & &
          \\[-2ex]
          $7$ & - & $\ZZ_3$ & $\ZZ^{20}$ & - & - &  - \\
          \hline
          & & & & & &
          \\[-2ex]
          $8$  & - & - & $\ZZ^{132}$ & - & - & - \\
          \hline
          & & & & & &
          \\[-2ex]
          $9$  & - & - & $\ZZ_3^8 \oplus \ZZ^{42}$ &
          $\ZZ^{70}$ & - & - \\
          \hline
          & & & & & &
          \\[-2ex]
          $10$  & - & - & $\ZZ_3$ &
          $\ZZ^{1216}$ & - & - \\
          \hline
          & & & & & &
          \\[-2ex]
          $11$  & - & - & - & $\ZZ_3^{45} \oplus
          \ZZ^{1188}$ & $\ZZ^{252}$ & -  \\
          \hline
          & & & & & &
          \\[-2ex]
          $12$  & - & - & - & $\ZZ_3^{56}$  &
          $\ZZ^{12440}$& - \\
          \hline
          & & & & & &
          \\[-2ex]
          $13$  & - & - & - & $\ZZ_3$  &
          $T_1 \oplus \ZZ^{24596}$& $\ZZ^{924}$ \\
          \hline
          & & & & & &
          \\[-2ex]
          $14$  & - & - & - & -  &
          $T_2$ &  $\ZZ^{138048}$ \\
          \hline
        \end{tabular}
    \end{center}
    \end{footnotesize}
    \label{matching-fig}
  \end{table}

  For the sake of generality, we describe our simple
  representation-theoretic construction in terms of
  an arbitrary finite group acting on a chain complex of
  abelian groups; see
  Section~\ref{homology-sec}. The particular case that we are
  interested in is discussed in Section~\ref{match-sec}.
  In Section~\ref{further-sec}, we make some remarks and discuss 
  potential improvements and generalizations of our result.

  \section{Group actions on chain complexes}
  \label{homology-sec}
  
  We recall some elementary properties of group actions on 
  chain complexes; 
  see Bredon~\cite{Bredon} for a more thorough treatment.
  Let
  \[
  \begin{CD}
    \mathcal{C} :  \cdots @>{\partial_{d+1}}>> C_d 
    @>{\partial_d}>> C_{d-1} @>{\partial_{d-1}}>> C_{d-2} 
    @>{\partial_{d-2}}>> \cdots
  \end{CD}
  \]
  be a chain complex of abelian groups. 
  Let $G$ be a group acting 
  on $\mathcal{C}$, meaning the following for each $k \in \ZZ$:
  \begin{itemize}
  \item
    Every $g \in G$ defines a degree-preserving automorphism on
    $\mathcal{C}$.
  \item
    For every $g,h \in G$ and $c \in C_k$, we have that
    $g(h(c)) = (gh)(c)$.
  \item
    For every $g \in G$ and $c \in C_k$, we have that
    $\partial_k(g(c)) = g(\partial_k(c))$.
  \end{itemize}

  Let $C_d^G$ be the subgroup of $C_d$ generated by $\{c
  - g(c) : c 
  \in C_d, g \in G\}$ and let $\mathcal{C}^G$ be the corresponding
  chain complex. $\mathcal{C}^G$ is indeed a chain complex, because
  $\partial_d(c-g(c)) = \partial_d(c)-g(\partial_d(c)) \in C^G_{d-1}$
  whenever $c \in C_d$.
  Writing $C_d/G = C_d/C_d^G$, we obtain the
  quotient chain complex
  \[
  \begin{CD}
    \mathcal{C}/G :  \cdots @>{\partial_{d+1}}>> C_d/G
    @>{\partial_d}>> C_{d-1}/G @>{\partial_{d-1}}>> C_{d-2}/G 
    @>{\partial_{d-2}}>> \cdots
  \end{CD}
  \]
  In particular, we have the following exact sequence of homo\-logy
  groups for each $d$: 
  \[
  \begin{CD}
    H_{d+1}(\mathcal{C}/G) 
    \longrightarrow
    H_d(\mathcal{C}^G) 
    \longrightarrow
    H_d(\mathcal{C})
    @>{\pi^*_d}>>
    H_d(\mathcal{C}/G) 
    \longrightarrow
    H_{d-1}(\mathcal{C}^G);
  \end{CD}  
  \]
  $\pi^*_d$ is the map induced by the natural projection map $\pi_d :
  C_d \rightarrow C_d/G$.

  From now on, assume that $G$ is finite.
  For an element $c \in C_d$, let $[c]$ denote the corresponding
  element in $C_d/G$; $[c] = c + C_d^G$.
  Define $\groupsum{G}{c} = \sum_{g \in G} g(c)$. Clearly, 
  $\groupsum{G}{c} = 0$ for all $c \in C_d^G$ and
  $\groupsumvoid{G}$ commutes with $\partial_d$.
  Let $\varphi_d : C_d/G \rightarrow C_d$ be the homomorphism defined
  by $\varphi_d([c]) = \groupsum{G}{c}$. Since $\groupsumvoid{G}$
  vanishes on $C_d^G$, we have that $\varphi_d$ is well-defined.
  Moreover,  $\partial_d \circ \varphi_d = \varphi_{d-1} \circ
  \partial_d$, because 
  \[
  \partial_d(\varphi_d([c])) = \partial_d(\groupsum{G}{c}) =
  \groupsum{G}{\partial_d(c)} 
  = \varphi_{d-1}([\partial_d(c)]) 
  = \varphi_{d-1}(\partial_d([c])).
  \]

  Let $\varphi^*_d : H_d(\mathcal{C}/G)
  \rightarrow H_d(\mathcal{C})$ be the map induced by 
  $\varphi_d$; this is a well-defined homomorphism
  by the above discussion. 
  \begin{lemma}
    The kernel of $\varphi^*_d$ has finite exponent dividing $|G|$.
    As a consequence,
    if the torsion subgroup of
    $H_d(\mathcal{C})$ has finite exponent $e$ ($e=1$ if
    there is no torsion), then 
    the exponent of the torsion subgroup of $H_d(\mathcal{C}/G)$ is
    also finite and divides $|G|\cdot e$. 
    \label{Gtimesdsum-lem}  
  \end{lemma}
  \begin{proof}
    Let $c \in H_d(\mathcal{C})$. Since 
    $[\varphi^*_d([c])] = [\groupsum{G}{c}] = |G|\cdot [c]$
    and $0 = [e\cdot \varphi^*_d([c])] = e\cdot |G|\cdot [c]$, we are
    done.
  \end{proof}

  \section{Detecting $5$-torsion in the homology of $\M[14]$}
  \label{match-sec}

  Let $\lambda = (\lambda_1, \ldots, \lambda_n)$ be a sequence of
  nonnegative integers summing to $N$. Define
  $\Symm{\lambda}$ to be the Young group $\Symm{\lambda_1} \times \cdots
  \times \Symm{\lambda_n}$. Write $[N]$ as a disjoint
  union $\bigcup_{i=1}^n U_i$ such that $|U_i| = \lambda_i$ for each
  $i$ and let $\Symm{\lambda_i}$ act on $U_i$ in the natural manner
  for each $i$. This yields an action of $\Symm{\lambda}$ on $[N]$,
  and this action induces an action on 
  the chain complex $\tilde{\mathcal{C}}(\M[N])$. In particular,
  we have the following result:
  \begin{lemma}
    \label{matchexp-lem}
    Let 
    $\varphi^*_d : H_d(\tilde{\mathcal{C}}(\M[N])/\Symm{\lambda})
    \rightarrow \tilde{H}_d(\M[N])$ be defined as in 
    Lemma~{\rm\ref{Gtimesdsum-lem}}.
    Then the kernel of $\varphi^*_d$ has finite exponent 
    dividing $\prod_{i=1}^n \lambda!$.
  \end{lemma}
  \begin{proof}
    This is an immediate consequence of Lemma~\ref{Gtimesdsum-lem}.
  \end{proof}
  
  Let $\Delta_\lambda$ be the subfamily of $\M[N]$ consisting of 
  all $\sigma$ such that there are two distinct edges $ab$ and $cd$ in
  $\sigma$ with the property that $\{a,c\} \subseteq U_i$ and $\{b,d\}
  \subseteq U_j$ for some $i$ and $j$ (possibly equal).
  Write $\Gamma_\lambda = \M[N] \setminus \Delta_\lambda$; this is a
  simplicial complex.
  Define $\kappa : [N] \rightarrow [n]$ by $\kappa^{-1}(\{i\}) =
  U_i$. Extend $\kappa$ to $\Gamma_\lambda$ by defining
  \[
  \kappa(\{a_1b_1, \ldots, a_rb_r\}) = 
  \{\kappa(a_1)\kappa(b_1), \ldots, \kappa(a_r)\kappa(b_r)\}.
  \]
  \begin{lemma}
    \label{wellkappa-lem}
    We have that $\kappa$ is a dimension-preserving surjective map
    from $\Gamma_\lambda$ to $\BD{n}{\lambda}$.
  \end{lemma}
  \begin{proof}
    To see that $\kappa$ is dimension-preserving, 
    note that $|\kappa(\sigma)| = |\sigma|$
    whenever $\sigma$ belongs to $\Gamma_\lambda$. 
    Namely, there are no multiple edges or multiple loops
    in $\kappa(\sigma)$ by definition of $\Gamma_\lambda$.
    Moreover, $\kappa(\sigma)$ belongs to $\BD{n}{\lambda}$,
    because for each $i \in [n]$, the degree in $\kappa(\sigma)$ of the
    vertex $i$ equals the sum of the degrees 
    in $\sigma$ of all vertices in $U_i$; this is at most
    $|U_i| = \lambda_i$.
    To prove surjectivity, use a simple induction argument over
    $\lambda$; remove one edge at a time from a given graph in
    $\BD{n}{\lambda}$.  
  \end{proof}

  \begin{lemma}
    For $\sigma, \tau \in \Gamma_\lambda$, we have that
    $\kappa(\sigma) = \kappa(\tau)$ if and only if  
    there is a $g$ in $\Symm{\lambda}$ such that $g(\sigma) = \tau$.
    \label{kappa-lem}
  \end{lemma}
  \begin{proof}
    Clearly, $\kappa(\sigma) = \kappa(g(\sigma))$ for all $\sigma
    \in \Gamma_\lambda$ and $g \in \Symm{\lambda}$. For the other
    direction, 
    write $\sigma = \{a_jb_j : j \in J\}$ and
    $\tau = \{a'_jb'_j : j \in J\}$, where 
    $\kappa(a_j) = \kappa(a'_j)$ and 
    $\kappa(b_j) = \kappa(b'_j)$ for each $j \in J$.
    Define $g(a_j) = a'_j$ and $g(b_j) = b'_j$ for each $j \in J$ and
    extend $g$ to a permutation on $[N]$ such that $g(U_i) = U_i$ for
    each $i \in [n]$. One easily checks that $g$ has the desired
    properties.
  \end{proof}

  For a set $\sigma$, we let $\oriented{\sigma}$ denote the oriented
  simplex corresponding to $\sigma$; fixing an order of the elements
  in $\sigma$, this is well-defined.
  Given an oriented simplex $\oriented{\sigma} = a_1b_1
  \wedge \cdots \wedge a_rb_r$, we define 
  \[
  \kappa(\oriented{\sigma}) = \kappa(a_1)\kappa(b_1)\wedge \cdots
  \wedge \kappa(a_r)\kappa(b_r), 
  \]
  thereby preserving orientation. Extend $\kappa$ linearly to a
  homomorphism $\tilde{\mathcal{C}}(\Gamma_\lambda) \rightarrow
  \tilde{\mathcal{C}}(\BD{n}{\lambda})$.
  \begin{lemma}
    \label{hatkappa-lem}
    The map $\hat{\kappa} :
    \tilde{\mathcal{C}}(\Gamma_\lambda)/\Symm{\lambda} \rightarrow
    \tilde{\mathcal{C}}(\BD{n}{\lambda})$ defined as
    $\hat{\kappa}([c]) = \kappa(c)$ 
    is a chain complex isomorphism.
  \end{lemma}
  \begin{proof}
    First of all, one easily checks that $\hat{\kappa}$ is
    well-defined and commutes with the boundary operator; for the
    former property, note that
    $\kappa(\oriented{\sigma}) = \kappa(g(\oriented{\sigma}))$  
    for all $g \in \Symm{\lambda}$ and $\sigma \in
    \Gamma_\lambda$. Moreover, $\hat{\kappa}$ is surjective, 
    because $\kappa$ is surjective by Lemma~\ref{wellkappa-lem}. Finally,
    to see that $\hat{\kappa}$ is 
    injective, define $\mu : \tilde{\mathcal{C}}(\BD{n}{\lambda})
    \rightarrow \tilde{\mathcal{C}}(\Gamma_\lambda)/\Symm{\lambda}$ 
    as $\mu(c') = [c]$, where $c$ is any element
    in $\tilde{\mathcal{C}}(\Gamma_\lambda)$
    such that $\kappa(c) = c'$; this is well-defined by
    Lemma~\ref{kappa-lem}. Since 
    $\mu \circ \hat{\kappa}([c]) = \mu(\kappa(c)) 
    = [c]$,
    injectivity follows.
  \end{proof}
  
  \begin{theorem}
    \label{split-thm}
    We have the chain complex isomorphism
    \[
    \tilde{\mathcal{C}}(\M[N])/\Symm{\lambda} 
    \cong
    \tilde{\mathcal{C}}(\BD{n}{\lambda}) \oplus
    \tilde{\mathcal{C}}(\Delta_\lambda)/\Symm{\lambda}.
    \]
  \end{theorem}
  \begin{proof}
    By Lemma~\ref{hatkappa-lem}, it suffices to prove that
    \[
    \tilde{\mathcal{C}}(\M[N])/\Symm{\lambda} 
    \cong
    \tilde{\mathcal{C}}(\Gamma_\lambda)/\Symm{\lambda} \oplus
    \tilde{\mathcal{C}}(\Delta_\lambda)/\Symm{\lambda}.
    \]
    Clearly, the boundary in
    $\tilde{\mathcal{C}}(\M[N])/\Symm{\lambda}$ of any element in
    $\tilde{\mathcal{C}}(\Gamma_\lambda)/\Symm{\lambda}$ is again an
    element in $\tilde{\mathcal{C}}(\Gamma_\lambda)/\Symm{\lambda}$, 
    $\Gamma_\lambda$ being a subcomplex of $\M[N]$.
    It remains to prove that
    $[\partial(\oriented{\sigma})]  
    \in \tilde{\mathcal{C}}(\Delta_\lambda)/\Symm{\lambda}$
    for each $\sigma \in \Delta_\lambda$. 
    Write $\oriented{\sigma} = a_1b_1 \wedge a_2b_2 \wedge
    \oriented{\tau}$, where
    $a_1,a_2 \in U_i$ and $b_1, b_2 \in U_j$ for some 
    $i$ and $j$.
        We obtain that
    \[
    \partial(\oriented{\sigma})
    = a_2b_2 \wedge \oriented{\tau} - a_1b_1 \wedge \oriented{\tau}
    + a_1b_1\wedge a_2b_2 \wedge \partial(\oriented{\tau}).
    \]
    Since the group element $(a_1,a_2)(b_1,b_2)$ belongs to
    $\Symm{\lambda}$ and transforms $a_1b_1 \wedge \oriented{\tau}$
    into $a_2b_2 \wedge \oriented{\tau}$, it follows that
    \[
    \partial([\oriented{\sigma}])
    = \left[a_1b_1\wedge a_2b_2 \wedge
    \partial(\oriented{\tau})\right],
    \]
    which is indeed an element in 
    $\tilde{\mathcal{C}}(\Delta_\lambda)/\Symm{\lambda}$.
  \end{proof}

  \noindent
  \emph{Remark.}
  One may note that
  $\tilde{\mathcal{C}}(\Delta_\lambda)/\Symm{\lambda}$ is a chain
  complex of elementary $2$-groups. Namely, with notation as in the
  above proof,  
  we have that $(a_1,a_2)(b_1,b_2)$ maps $\oriented{\sigma} =
  a_1b_1\wedge a_2b_2 \wedge \oriented{\tau}$ to $a_2b_2\wedge a_1b_1
  \wedge \oriented{\tau} = -\oriented{\sigma}$.
  As a consequence, 
  $[\oriented{\sigma}] = -[\oriented{\sigma}]$, which
  implies that $2[\oriented{\sigma}] = 0$.

  \begin{theorem}
    There is a homomorphism
    $\tilde{H}_d(\BD{n}{\lambda}) \rightarrow
    \tilde{H}_d(\M[N])$ 
    such that the kernel has finite exponent dividing 
    $\prod_{i=1}^n \lambda!$.
    \label{mbd-thm}
  \end{theorem}
  \begin{proof}
    This follows immediately from Lemma~\ref{matchexp-lem} and
    Theorem~\ref{split-thm}.
  \end{proof}
  Let us summarize the situation.
  \begin{corollary}
    We have a long exact sequence
    \[
    \begin{CD}
      & & & \cdots &
      @>>>  
      H_{d+1}(\tilde{\mathcal{C}}(\M[N])/\Symm{\lambda}) 
      \\ 
      @>>>  
      H_d(\tilde{\mathcal{C}}(\M[N])^{\Symm{\lambda}})
      @>>>  
      \tilde{H}_d(\M[N])
      @>\pi^*_d>>  
      H_d(\tilde{\mathcal{C}}(\M[N])/\Symm{\lambda})
      \\
      @>>>  
      H_{d-1}(\tilde{\mathcal{C}}_d(\M[N])^{\Symm{\lambda}})
      @>>> 
      \cdots ,
    \end{CD}
    \]
    where
    \[
    H_{d}(\tilde{\mathcal{C}}(\M[N])/\Symm{\lambda}) 
    \cong
    \tilde{H}_d(\BD{n}{\lambda}) \oplus
    H_d(\tilde{\mathcal{C}}(\Delta_\lambda)/\Symm{\lambda})
    \]
    and 
    $\pi^*_d$ has an ``inverse'' $\varphi^*_d$ satisfying
    $\pi^*_d \circ \varphi^*_d = \prod_{i=1}^n \lambda! \cdot {\rm id}$.
    In particular, if $\prod_{i=1}^n \lambda!$ is a
    unit in the underlying coefficient ring, then 
    \[
    \begin{CD}
      \tilde{H}_d(\M[N]) \cong
      \tilde{H}_d(\BD{n}{\lambda}) \oplus
      H_d(\tilde{\mathcal{C}}(\M[N])^{\Symm{\lambda}}).
    \end{CD}
    \]
    \label{mbdexact-cor}
 \end{corollary}
  For the final statement, note that 
  $\tilde{\mathcal{C}}(\Delta_\lambda)/\Symm{\lambda}$
  is zero if $2$ is a unit in the underlying coefficient ring or if
  $\lambda = (1, \ldots, 1)$.

  \begin{proof}[Proof of Theorem~{\rm\ref{main-thm}}]
    By Theorem~\ref{ShWa-thm}, we already know that there are elements
    of order three in $\tilde{H}_{4}(\M[14];\ZZ)$
    and that the group is finite.
    Applying Theorem~\ref{andersen-thm}, we obtain that the exponent
    of $\tilde{H}_4(\BD{7}{2};\ZZ)$ is five. 
    Selecting $\lambda = (2,2,2,2,2,2,2)$ and noting that
    $\prod_i \lambda_i! = 128$ and $\gcd(5,128) = 1$, we are done 
    by Theorem~\ref{mbd-thm}.
  \end{proof}

  Let $(2^a1^b)$ denote the sequence consisting of $a$ occurrences of
  the value 2 and $b$ occurrences of the value 1.
  One may try to obtain further information about the homology of
  $\M[14]$ by computing the homology of $\BD{14-a}{(2^a1^{14-2a})}$
  for $a \le 6$.
  The ideal, of course, would be to compute the homology of $\M[14]$
  directly, but this appears to be beyond the capacity of today's
  (standard) computers. Using the computer program {\tt 
  CHomP} \cite{Pilar}, we managed to compute the
  $\ZZ_p$-homology of $\BD{8}{(2^6 1^2)}$ for $p \in \{2,3,5\}$, and
  the results suggest that $\tilde{H}_4(\BD{8}{(2^6 1^2)};
  \ZZ) \cong \tilde{H}_4(\BD{7}{2};\ZZ) \cong \ZZ_5$. In particular,
  it seems that we cannot gather any  
  additional information about the homology of $\M[14]$ from that
  of $\BD{8}{(2^61^2)}$. 

  Via a calculation with the {\tt Homology}
  computer program \cite{Homoprog}, we discovered that 
  \[
  \tilde{H}_4(\BD{11}{(2^2 1^9)}; \ZZ) \cong \ZZ_3^{10} \oplus
  \ZZ^{6142}.
  \]
  By Theorem~\ref{mbd-thm} and well-known properties of the 
  rational homology of $\M[13]$ \cite{Bouc}, this yields the following
  result:
  \begin{proposition}
    We have that $\tilde{H}_{4}(\M[13];\ZZ) \cong T \oplus
    \ZZ^{24596}$, where $T$ is a finite group containing 
    $\ZZ_3^{10}$ as a subgroup.
    \label{m13-prop}
  \end{proposition}
  See Tables~\ref{matchbd1-fig} and \ref{matchbd2-fig} for more
  information about torsion in the homology of 
  $\BD{a+b}{2^a1^{b}}$ for small values of $a$ and $b$.
  The numerical data in Table~\ref{matchbd1-fig} suggests that the
  Sylow $3$-subgroup of $\tilde{H}_{(n-5)/2}(\BD{n-a}{2^a1^{n-2a}};
  \ZZ)$ is an elementary $3$-group of rank $\binom{n-a-1}{(n+5)/2}$.
  
  \begin{table}
    \caption{Torsion subgroup of $\tilde{H}_i(\BD{n-a}{2^a1^{n-2a}}; \ZZ)$
      for $n = 2i+5$.} 
    \begin{footnotesize}
    \begin{center}
        \begin{tabular}{|r||c|c|c|c|c|c|c|c|}
          \hline
          & & & & & & & &
          \\[-1.5ex]
          &  $a=0$ & \ 1 \ & \ 2 \ & \ 3 \ & \ 4 \ & \ 5 \ & \ 6\ & \ 7 \
          \\
          \hline
          \hline
          & & & & & & & &
          \\[-2ex]
          $n = 3$ & $0$ & $0$ & - & - & - & - & - & - \\
          \hline
          & & & & & & & &
          \\[-2ex]
          $5$ & $0$ & $0$ & $0$ & - & - & - & - & - \\
          \hline
          & & & & & & & &
          \\[-2ex]
          $7$ & $\ZZ_3$ & $0$ & $0$ & $0$ & - & - & - & - \\
          \hline
          & & & & & & & &
          \\[-2ex]
          $9$ & $\ZZ_3^8$ & $\ZZ_3$ & $0$ & $0$ & $0$ & - & - & - \\
          \hline
          & & & & & & & &
          \\[-2ex]
          $11$ & $\ZZ_3^{45}$ & $\ZZ_3^9$ & $\ZZ_3$ & $0$ & $0$ &  $0$
          & - & - \\ 
          \hline
          & & & & & & & &
          \\[-2ex]
          $13$  & ? & ? & $\ZZ_3^{10}$ & $\ZZ_3$ & $0$ & $0$ & $0$ & - \\
          \hline
          & & & & & & & &
          \\[-2ex]
          $15$  & ? & ? & ? & ? & ? & $\ZZ_2$ & $0$ & $0$ \\
          \hline
        \end{tabular}
    \end{center}
    \end{footnotesize}
    \label{matchbd1-fig}
  \end{table}

  \begin{table}
    \caption{Torsion subgroup of $\tilde{H}_i(\BD{n-a}{2^a1^{n-2a}}; \ZZ)$
      for $n = 2i+6$.} 
    \begin{footnotesize}
    \begin{center}
        \begin{tabular}{|r||c|c|c|c|c|c|c|c|}
          \hline
          & & & & & & & &
          \\[-1.5ex]
          &  $a=0$ & \ 1 \ & \ 2 \ & \ 3 \ & \ 4 \ & \ 5 \ & \ 6\ & \ 7 \
          \\
          \hline
          \hline
          & & & & & & & &
          \\[-2ex]
          $n = 2$ & $0$ & $0$ & - & - & - & - & - & - \\
          \hline
          & & & & & & & &
          \\[-2ex]
          $4$ & $0$ & $0$ & $0$ & - & - & - & - & - \\
          \hline
          & & & & & & & &
          \\[-2ex]
          $6$ & $0$ & $0$ & $0$ & $0$ & - & - & - & - \\
          \hline
          & & & & & & & &
          \\[-2ex]
          $8$ & $0$ & $0$ & $0$ & $0$ & $0$ & - & - & - \\
          \hline
          & & & & & & & &
          \\[-2ex]
          $10$ & $\ZZ_3$ & $0$ & $0$ & $0$ & $0$ &  $0$
          & - & - \\ 
          \hline
          & & & & & & & &
          \\[-2ex]
          $12$  & $\ZZ_3^{56}$ & $\ZZ_3^{10}$ & $\ZZ_3$ & $0$ & $0$ & $0$ & $0$ & - \\
          \hline
          & & & & & & & &
          \\[-2ex]
          $14$  & ? & ? & ? & ? & ? & ? & $\ZZ_5$? & $\ZZ_5$ \\
          \hline
        \end{tabular}
    \end{center}
    \end{footnotesize}
    \label{matchbd2-fig}
  \end{table}

  \section{Remarks and further directions}
  \label{further-sec}

  Using {\tt CHomP} \cite{Pilar}, we managed to
  compute a generator 
  $\gamma'$ for the homology group
  $\tilde{H}_4(\BD{7}{2};\ZZ)\cong \ZZ_5$;
  \begin{eqnarray*}
  \gamma'\!  &=&
  ([12,45,23] + [12,23,34] + [12,34,15] + [12,15,33] + [12,33,45] \\
  &+&
  [22,33,15] + [22,15,34] + [22,34,11] + [22,11,45] + [22,45,33] \\
  &+&
  [11,23,45] + [11,34,23]) \wedge (46-66) \wedge (57-77);
  \end{eqnarray*}
  $[ab,cd,ef] = ab \wedge cd \wedge ef$.
  Note that $\gamma' = \gamma/{\Symm{(2^7)}}$, where
  \begin{eqnarray*}
  \gamma &=&
  ([1\hat{2},5\hat{4},2\hat{3}] + [1\hat{2},2\hat{3},3\hat{4}] + [1\hat{2},3\hat{4},5\hat{1}] + [1\hat{2},5\hat{1},3\hat{3}] + [1\hat{2},3\hat{3},5\hat{4}] \\
  &+&
  [2\hat{2},3\hat{3},5\hat{1}] + [2\hat{2},5\hat{1},3\hat{4}] + [2\hat{2},3\hat{4},1\hat{1}] + [2\hat{2},1\hat{1},5\hat{4}] + [2\hat{2},5\hat{4},3\hat{3}] \\
  &+&
  [1\hat{1},2\hat{3},5\hat{4}] + [1\hat{1},3\hat{4},2\hat{3}]) \wedge
  (4\hat{6}-6\hat{6}) \wedge (7\hat{5}-7\hat{7}). 
  \end{eqnarray*}
  Here, $\hat{i}$ denotes the vertex $i+7$, 
  and the group action is given by the partition $\{U_i : i \in [7]\}$,
  where $U_i = \{i,\hat{i}\}$. Since $\tilde{H}_4(\M[14];\ZZ)$ is
  finite, we conclude that
  $\gamma$ has finite exponent a multiple of five
  in $\tilde{H}_4(\M[14];\ZZ)$. Note that we may view
  $\gamma$ as the product of one cycle in $\tilde{H}_2(\M[8];\ZZ)$
  and two cycles in $\tilde{H}_0(\M[3];\ZZ)$ (defined on three
  disjoint vertex sets).

  Another observation is that we have the following portion of the
  long exact sequence for the pair $(\M[14],\M[13])$:
  \[
  \begin{CD}
    \tilde{H}_4(\M[13]; \ZZ) @>>>
    \tilde{H}_4(\M[14]; \ZZ) @>>>
    \bigoplus_{13} \tilde{H}_3(\M[12]; \ZZ);
 \end{CD}
  \]
  see Bouc \cite{Bouc}. 
  Since $\tilde{H}_3(\M[12]; \ZZ)$ is an elementary $3$-group by
  the data in Table~\ref{matching-fig}, this yields that  
  there must be some element $\delta$ in $\tilde{H}_4(\M[13];\ZZ)$
  such that $\delta$ is identical in  $\tilde{H}_4(\M[14];\ZZ)$ to
  $\gamma$ or $3\gamma$. Obviously, the exponent of $\delta$  
  in $\tilde{H}_4(\M[13];\ZZ)$ is either infinite or a nonzero
  multiple of five; we conjecture the former.

  As mentioned in Section~\ref{intro-sec}, we do not know whether
  there is $5$-torsion in the homology of $\M[n]$ when $n \ge 13$ and
  $n \neq 14$. We would indeed have such torsion for all even $n \ge
  16$ if
  \begin{equation}
    \tilde{H}_d(\M[n]; \ZZ)
    \cong
    \tilde{H}_d(\M[n] \setminus e; \ZZ) \oplus
    \tilde{H}_{d-1}(\M[n-2]; \ZZ)
    \label{onestepdecision-eq}
  \end{equation}
  for all even $n \ge 16$ and $d = n/2-3$.
  Here, $e$ is the edge between $n-1$ and $n$ and $\M[n]
  \setminus e$ is the complex obtained from $\M[n]$ by removing the
  $0$-cell $e$. Using computer, we have verified
  (\ref{onestepdecision-eq}) for all 
  $(n,d)$ such that $n \le 11$ and $d \ge 0$. Note that
  (\ref{onestepdecision-eq}) would follow if the
  sequence 
  \[
  \begin{CD}
    0 \longrightarrow
    \tilde{H}_d(\M[n] \setminus e; \ZZ) @>>>
    \tilde{H}_d(\M[n]; \ZZ)  @>>>
    \tilde{H}_{d-1}(\M[n-2]; \ZZ)  \longrightarrow
    0
 \end{CD}
  \]
  turned out to be split exact. By the long exact sequence for the
  pair $(\M[n],\M[n] \setminus e)$, the mid-portion of this sequence
  is indeed exact.

\end{document}